\documentclass[11pt,a4paper]{amsart}
\usepackage[utf8]{inputenc}
\usepackage[english]{babel}
\usepackage{amsmath}
\usepackage{amsfonts}
\usepackage{amssymb}
\usepackage{graphicx, overpic}
\usepackage{amsthm}
\usepackage{textcomp}
\usepackage{mathrsfs}
\usepackage{hyperref}
\usepackage{enumitem}

\usepackage[margin=1.25in]{geometry}

\theoremstyle{plain}
\newtheorem{thm}{Theorem}[section]
\newtheorem{lem}[thm]{Lemma}
\newtheorem{prop}[thm]{Proposition}
\newtheorem{cor}[thm]{Corollary}
\newcounter{theoremalph}

\theoremstyle{definition}
\newtheorem{defn}[thm]{Definition}
\newtheorem{example}[thm]{Example}
\newtheorem{rem}[thm]{Remark}

\newtheorem{obs}[thm]{Observation}
\newtheorem{claim}[thm]{Claim}

\newtheorem{notation}[thm]{Notation}

\newcommand{\R}{\mathbb{R}}
\newcommand{\N}{\mathbb{N}}

\newcommand{\cC}{\mathcal{C}}

\newcommand{\cG}{\mathcal{G}}

\newcommand{\la}{\left\langle}
\newcommand{\ra}{\right\rangle}

\makeatletter
\newsavebox{\@brx}
\newcommand{\llangle}[1][]{\savebox{\@brx}{\(\m@th{#1\langle}\)}%
  \mathopen{\copy\@brx\kern-0.5\wd\@brx\usebox{\@brx}}}
\newcommand{\rrangle}[1][]{\savebox{\@brx}{\(\m@th{#1\rangle}\)}%
  \mathclose{\copy\@brx\kern-0.5\wd\@brx\usebox{\@brx}}}
\makeatother

\usepackage{mathtools}

\usepackage{tikz-cd} %
\tikzset{
	labl/.style={anchor=south, rotate=90, inner sep=.5mm}
}

\definecolor{amethyst}{rgb}{0.6, 0.4, 0.8}

\newcommand{\hide}[1]{}

\title{Failure of quasi-isometric rigidity for infinite-ended groups}

\author{Nir Lazarovich}
\author{Emily Stark}

\date{\today}

\begin{document}
  
  \begin{abstract}
    We prove that an infinite-ended group whose one-ended factors have finite-index subgroups and are in a family of groups with a nonzero multiplicative invariant is not quasi-isometrically rigid. 
    Combining this result with work of the first author proves that a residually-finite multi-ended hyperbolic group is quasi-isometrically rigid if and only if it is virtually free.
    The proof adapts an argument of Whyte for commensurability of free products of closed hyperbolic surface groups. 
  \end{abstract}

 \maketitle
 
\section{Introduction}


\vskip.2in

Infinite-ended groups display a strong form of quasi-isometric flexibility: Papasoglu--Whyte~\cite{papasogluwhyte} proved the quasi-isometry type of an infinite-ended group depends only on the quasi-isometry types of its one-ended factors. Commensurability classes, on the other hand, are often finer. For example, while there is one abstract commensurability class of closed hyperbolic surface groups, there are infinitely many abstract commensurability classes of their free products. Indeed, Whyte~\cite[Theorem 1.6]{whyte1999amenability} proved that free products of two hyperbolic surface groups are abstractly commensurable if and only if they have equal Euler characteristic. 

We adapt Whyte's argument to free products of groups with a nonzero multiplicative invariant, like Euler characteristic. Consequently, we show quasi-isometric rigidity fails for many infinite-ended groups, where a finitely generated group $G$ is {\it quasi-isometrically rigid} if any group quasi-isometric to~$G$ is abstractly commensurable to $G$, meaning the groups contain isomorphic finite-index subgroups. In particular, by~\cite{lazarovich23}, our main result gives the following:

\begin{thm} \label{thm:introhyp}
 A residually-finite multi-ended hyperbolic group is quasi-isometrically rigid if and only if it is virtually free.
\end{thm}

Whether all hyperbolic groups are residually finite is an open question of considerable interest in the field. 
We remark that in the proof of Theorem~\ref{thm:introhyp} residual-finiteness is only used to ensure that the one-ended factors of the multi-ended hyperbolic group have some proper finite-index subgroup. 


Quasi-isometric rigidity is a driving problem in geometric group theory.  Many well-studied groups have been shown to be quasi-isometrically rigid, including non-abelian free groups~\cite{stallings,dunwoody1985accessibility}, fundamental groups of closed hyperbolic surfaces~\cite{tukia, gabai, cassonjungreis}, non-uniform lattices in rank-one symmetric spaces~\cite{Schwartz95}, mapping class groups~\cite{behrstockkleinerminskymosher}, and uniform lattices in certain hyperbolic buildings~\cite{bourdonpajot, haglund06, xie06}. Fundamental groups of closed hyperbolic $3$-manifolds and free products of closed hyperbolic surface groups are non-examples, and this paper gives many more. We refer the reader to the book of Drutu--Kapovich~\cite{drutukapovich} for additional background. 


  We show quasi-isometric rigidity for infinite-ended groups fails in the presense of a suitable \emph{multiplicative invariant} (also known as \emph{volume} cf. \cite{reznikov}). 

\begin{defn}
    Let $\cC$ be a family of groups closed under taking finite-index subgroups. A {\it multiplicative invariant on $\cC$} is a function $\rho:\cC \rightarrow \R$ so that if $G \in \cC$ and $H$ is a finite-index subgroup of $G$, then $\rho(H) = [G:H]\rho(G)$.
\end{defn}

\begin{example} The following are examples of multiplicative invariants. 
    \begin{enumerate}
        \item Euler characteristic for the family groups with a finite classifying space; see discussion on generalizations in work of Chiswell~\cite{chiswell};
        \item $\ell^2$-Betti numbers for the family groups with a classifying space with finite skeleta in each dimension; for background, see work of L\"{u}ck~\cite{luck};
        \item Volume for the family of fundamental groups of closed hyperbolic manifolds of dimension $n \geq 3$ by Mostow Rigidity~\cite{mostow_strongrigidity}
        \item Volume from an upper or lower volume in the sense of Reznikov~\cite[Section~1]{reznikov}. Examples include deficiency volume and rank volume~\cite{reznikov} and a volume from topological complexity due to Lazarovich~\cite{lazarovich23}. 
    \end{enumerate}
\end{example}

Let $\cC$ be a family of one-ended groups and suppose that $\cC$ is closed under taking finite-index subgroups. Suppose $\cC$ has a multiplicative invariant that is nonzero for every element of $\cC$. 

\begin{thm} \label{thm:intromain}
   Let $G_i \in \cC$, and suppose $H_i \leq G_i$ is a proper finite-index subgroup for some $n>1$ and $1 \leq i \leq n$. Then, the groups $G_1* \ldots * G_n$ and $H_1 * \ldots *H_n$ are quasi-isometric and not abstractly commensurable. 
\end{thm}

\begin{cor} \label{main cor}
    If $G$ is a finitely presented infinite-ended group that is not virtually free and whose one-ended factors are in $\cC$ and contain non-trivial finite-index subgroups, then $G$ is not quasi-isometrically rigid.
\end{cor}

\begin{proof}[Proof of Theorem~\ref{thm:introhyp}]
 A multi-ended hyperbolic group which is not virtually free has finitely many one-ended factors. If the group is residually finite, then each factor is a residually-finite hyperbolic group and, in particular, has a proper finite-index subgroup. Every hyperbolic group has finite asymptotic dimension~\cite{roe}, and so there exists $m$ such that the asymptotic dimension of all the factors is bounded by $m$.
    Let $\cC$ be the family of one-ended hyperbolic groups whose asymptotic dimension is bounded by $m$. 
    By \cite[Proof of Theorem~A]{lazarovich23}, the family $\cC$ has a multiplicative invariant, namely $\underline{C}_{2,m}$, that is nonzero for every element of $\cC$.
    The conclusion now follows from Corollary~\ref{main cor}.
\end{proof}

The free products $G_1* \ldots * G_n$ and $H_1 * \ldots *H_n$ are quasi-isometric by work of Papasoglu--Whyte~\cite{papasogluwhyte} since $G_i$ and $H_i$ are quasi-isometric for each $i$. It remains to prove the groups are not abstractly commensurable. To do this, given a multiplicative invariant~$\rho$ for a family of groups $\cC$, we construct a multiplicative invarant $\chi_\rho$ on the family of free products whose non-free free factors are elements of $\cC$, and we adapt an argument of Whyte~\cite[Theorem 1.6]{whyte1999amenability}. 
We note that Sykiotis~\cite{sykiotis} defines a related multiplicative invariant for infinite-ended groups through a limiting process; our construction is simpler and more direct. 

To prove that two groups are quasi-isometric, often one finds geometric actions of the groups on the same proper metric space - in this case, the groups are said to have a {\it common model geometry}. Theorem~\ref{thm:intromain} can be used to exhibit quasi-isometric groups with no common model geometry. Indeed, such examples arise via finitely generated groups that are not quasi-isometrically rigid, but are {\it action rigid}, meaning that whenever the group and another finitely generated group act geometrically on the same proper metric space, the groups are abstractly commensurable. Many examples of action rigid free products are given in~\cite{starkwoodhouse19, mssw}, including free products of closed hyperbolic manifold groups.

\subsection*{Acknowledgements} 
    NL was supported by Israel Science Foundation (grant no. 1562/19).
    ES was supported by NSF Grant No. DMS-2204339.

    \section{Preliminaries} \label{sec:prelims}
 
    We will use the notion of a graph of spaces and a graph of groups; we refer the reader to~\cite{SerreTrees, ScottWall79} for additional background.

\begin{defn} (Graph notation.)
    A {\it graph} $\Gamma$ is a set of vertices $V\Gamma$ together with a set of edges $E\Gamma$. An edge $e$ is oriented with an initial vertex $\iota(e)$ and a terminal vertex~$\tau(e)$. For every edge $e$ there is an edge $\bar{e}$ with reversed orientation so that $\iota(e) = \tau(\bar{e})$. Further, $e \neq \bar{e}$ and $e = \bar{\bar{e}}$.     
\end{defn}

    We often suppress the orientation of an edge if it is not relevant to a construction. 

 \begin{defn} \label{defn:GoS}
A \emph{graph of spaces} $(X,\Gamma)$ consists of the following data:
\begin{enumerate}
    \item a graph $\Gamma$ called the {\it underlying graph}, 
	\item a connected space $X_v$ for each $v \in V\Gamma$ called a \emph{vertex space},
	\item a connected space $X_e$ for each $e\in E\Gamma$ called an \emph{edge space} such that $X_{\bar{e}} = X_e$, and
	\item a $\pi_1$-injective map $\phi_e : X_e \rightarrow X_{\tau(e)}$ for each $e\in E\Gamma$.
\end{enumerate}

The \emph{total space} $X$ is the following quotient space:
\[
 X = \left( \bigsqcup_{v \in V \Gamma} X_v \bigsqcup_{e \in E\Gamma} \bigl( X_e \times [0,1]\bigr) \right)  \; \Big/ \; \sim, 
\]
where $\sim$ is the relation that identifies $X_e \times \{0\}$ with $X_{\bar{e}} \times \{0\}$ via the identification of $X_e$ with $X_{\bar{e}}$, and the point $(x, 1)\in X_e \times [0,1]$ with $\phi_e(x)\in X_{\tau(e)}$.
\end{defn}

\begin{defn} \label{def:gog_gosp} 
      A \emph{graph of groups} $\mathcal{G}$ consists of the following data:
      \begin{enumerate}
      	\item a graph $\Gamma$ called the \emph{underlying graph},
          \item a group $G_v$ for each $v \in V\Gamma$ called a \emph{vertex group},
          \item  a group $G_e$ for each $e \in E\Gamma$ called an \emph{edge group}, such that $G_e = G_{\bar{e}}$,
          \item and injective homomorphisms $\Theta_e: G_e \rightarrow G_{\tau(e)}$ for each $e \in E\Gamma$ (called \emph{edge maps}).    
      \end{enumerate}
           If $\cG$ is a graph of groups with underlying graph $\Gamma$, a \emph{graph of spaces associated to  $\mathcal{G}$} consists of a graph of spaces $(X,\Gamma)$ with points $x_v \in X_v$ and $x_e \in X_e$ so that $\pi_1(X_v,x_v) \cong G_v$ and $\pi_1(X_e, x_e) \cong G_e$ for all $v \in V\Gamma$ and $e \in E\Gamma$, and the maps $\phi_e:X_e\rightarrow X_{\tau(e)}$ are such that $(\phi_e)_* = \Theta_e$.  
      The \emph{fundamental group} of the graph of groups $\mathcal{G}$ is $\pi_1(X)$. The fundamental group is independent of the choice of $X$ and is denoted by $\pi_1(\cG)$.
      \end{defn}

     If $G \cong \pi_1(\cG)$, where $\cG$ is a graph of groups, then we may view each vertex group of $\cG$ as a subgroup of $G$ (up to conjugacy). 
     Suppose that $G = \pi_1(\cG) =\pi_1(X)$, where $\cG$ is a graph of groups and $X$ is the corresponding total space. If $G' \leq G$, then $G' = \pi_1(X')$, where $X'$ covers $X$. This cover $q:X' \rightarrow X$ naturally yields a graph of spaces decomposition of $X'$ with underlying graph $\Gamma'$, where the edge and vertex spaces are components of the preimages of the edge and vertex spaces in~$X$. Hence, $G'$ is the fundamental group of a corresponding graph of groups $\cG'$. Each vertex group of $\cG'$ is then conjugate to a subgroup of a vertex group of $\cG$. 

\begin{defn}
    With the notation from the previous paragraph, if $G_v$ is a vertex group of~$\cG$, we say a vertex group $G_u'$ of $\cG'$ {\it covers} $G_v$ if $G_u'$ is conjugate to a subgroup of $G_v$ via the covering of the corresponding graphs of spaces. 
\end{defn}

We make use of the following observation. 
     
\begin{obs} \label{obs:multpreimage}
   Suppose $\cC$ is a family of groups closed with respect to taking finite-index subgroups and with a multiplicative invariant $\rho$. Suppose $\cG$ is a graph of groups with each vertex group in $\cC$ and $G = \pi_1(\cG)$. 
   Suppose $G'$ is an index-$k$ subgroup of $G$ and $\cG'$ is its corresponding graph of groups decomposition. Then for each vertex group $G_v \in \cG$, if $K = \{G'_{u_1}, \ldots, G'_{u_n}\}$ is the collection of vertex groups of $\cG'$ covering $G_v$, then \[\sum_{i=1}^n \rho(G'_{u_i}) = k \rho(G_v).\]  
\end{obs}

    \section{Multiplicative invariants and commensurability}

\begin{notation}
    For $r \geq 0$, let $F_r$ denote the free group of rank $r$. Throughout this section let $\cC$ be a family of one-ended groups that is closed with respect to taking finite-index subgroups and has a multiplicative invariant $\rho$. Let $\cC'$ denote the family of groups of the form $G = G_1 * \ldots * G_n * F_r$ for some $r \geq 0$, $n \geq 0$, and $G_i \in \cC$. Equivalently, $G \in \cC'$ if $G$ is the fundamental group of a graph of groups $\cG$ with trivial edge groups and each vertex group either trivial or in $\cC$. 
\end{notation}

    We will work in the setting of free products, rather than the general infinite-ended case, as the arguments are cleaner. See Remark~\ref{rem:infendedgen} regarding the more general setting.

\begin{defn}
    Let $\rho$ be a multiplicative invariant on the family of groups $\cC$.  Define a function $\chi_{\rho}:\cC' \rightarrow \R$  by
    \[ \chi_{\rho}(G_1*\ldots *G_n*F_r) = 1-r-n+ \sum_{i=1}^n \rho(G_i).\]
\end{defn}

\begin{rem}
    The function $\chi_{\rho}$ is a variant of the Euler characteristic. Viewing $G \in \cC'$ as the fundamental group of a graph of groups, each vertex in the graph of groups decomposition labeled $G_i \in \cC$ is assigned $\rho(G_i)$ rather than $1$. In particular, if the multiplicative invariant $\rho$ is Euler characteristic, then $\chi_{\rho}$ is as well. 
\end{rem}

To prove that $\chi_{\rho}$ is a multiplicative invariant, we will use the following. 

\begin{defn}
    Let $\cG$ be a graph of groups with underlying graph $\Gamma$, trivial edge groups, and with nontrivial vertex groups $G_1, \ldots, G_n \in \cC$. Define
    \[ \chi_{\rho}(\cG) = \chi(\Gamma) - n + \sum_{i=1}^n \rho(G_i). \]
\end{defn}

\begin{lem} \label{lemma:defs_agree}
    Let $G = G_1 * \ldots G_n * F_r \in \cC'$, and suppose $G \cong \pi_1(\cG)$, where $\cG$ is a graph of groups with underlying graph $\Gamma$, trivial edge groups, and with nontrivial vertex groups $G_1, \ldots, G_n \in \cC$. Then, $\chi_{\rho}(G) = \chi_{\rho}(\cG)$. 
\end{lem}
\begin{proof}
    By definition, 
    \begin{eqnarray*}
        \chi_{\rho}(G)  &=& 1-r-n + \sum_{i=1}^n\rho(G_i), \textrm{ and } \\
        \chi_{\rho}(\cG) &=& \chi(\Gamma) - n + \sum_{i=1}^n\rho(G_i). 
    \end{eqnarray*}
    Thus, we must show $\chi(\Gamma) = 1-r$. Indeed, it follows from the definition of the fundamental group of a graph of groups that $F_r \cong \pi_1(\Gamma)$. So, $\chi(\Gamma) = \chi(F_r) = 1-r$. 
\end{proof}

\begin{prop} \label{prop:prime_mult}
    The function $\chi_{\rho}:\cC' \rightarrow \R$ is a multiplicative invariant. 
\end{prop}
\begin{proof}
    Let $G = G_1 * \ldots * G_n * F_r \in \cC'$. Realize $G$ as the fundamental group of a graph of groups $\cG$ whose underlying graph $\Gamma$ has vertex set $V\Gamma = \{v_0, \ldots, v_n\}$ and (unoriented) edge set $E\Gamma = \{e_1, \ldots, e_n, a_1, \ldots, a_r\}$, where $e_i = \{v_0, v_i\}$ for $1 \leq i \leq n$, and $a_i$ is a loop based at $v_0$ for $1 \leq i \leq r$. Suppose each edge group is trivial, the vertex group $G_{v_0}$ is trivial, and the vertex group $G_{v_i}$ is $G_i$. 

    Let $G' \leq G$ be an index-$k$ subgroup. As discussed in Section~\ref{sec:prelims}, $G' \cong \pi_1(\cG')$, where $\cG'$ is a graph of groups with: 
    \begin{itemize}
        \item $k$ vertices with trivial vertex group that each cover $G_{v_0} = \la 1 \ra$,
        \item $k$ edges with trivial edge group that each cover $G_{a_i} = \la 1 \ra$ for $1 \leq i \leq r$,
        \item $k$ edges with trivial edge group that each cover $G_{e_i} = \la 1 \ra$ for $1 \leq i \leq n$, and
        \item for $i \in \{1, \ldots n \}$, there exists $n_i\in \N_{\geq 1}$ and 
        $n_i$ vertices $v_i^1, \ldots, v_i^{n_i}$ with vertex group $G_{v_i^j}$ that each cover $G_{v_i}$. Further, \[\sum_{j=1}^{n_i} \rho(G_{v_i^j}) = k\rho(G_i)\] by Observation~\ref{obs:multpreimage} since $\rho$ is a multiplicative invariant. 
    \end{itemize}
    Thus, $\chi_{\rho}(\cG') = k \chi_{\rho}(\cG)$.
    So, by Lemma~\ref{lemma:defs_agree}, $\chi_{\rho}(G') = \chi_{\rho}(\cG') = k \chi_{\rho}(\cG) = k \chi_{\rho}(G)$. 
\end{proof}

As explained in the introduction, Theorem~\ref{thm:intromain} follows from the next theorem. The proof is analogous to Whyte's argument in \cite[Theorem 1.6]{whyte1999amenability}.

\begin{thm} \label{thm:mainpaper}
    Let $\cC$ be a family of groups closed with respect to taking finite-index subgroups and with a multiplicative invariant  $\rho$ that is nonzero for every element of $\cC$. For $m>1$ and $1 \leq i \leq m$, let $G_i \in \cC$, and suppose $H_i \leq G_i$ is a proper finite-index subgroup. Then, the groups $G=G_1* \ldots * G_m$ and $H=H_1 * \ldots *H_m$ are not abstractly commensurable. 
\end{thm}
\begin{proof}
    Suppose towards a contradiction that there exist finite-index subgroups $G' \leq G$ and $H' \leq H$ of index $k$ and $k'$, respectively, so that $G' \cong H'$. Suppose that $H' \cong G' \cong G_1' *  \ldots*  G_n' * F_r$ where $r \geq 0$ and for each $i \in \{1, \ldots, n\}$, the group $G_i'$ covers $G_j$ and $H_{\ell}$ for some $j, \ell \in \{1,\ldots, m\}$. 
    By Proposition~\ref{prop:prime_mult}, $\chi_{\rho}(G') =  k \chi_{\rho}(G)= k' \chi_{\rho}(H)$. We will show that $k = k'$ by relating these numbers to $r,n$, and $m$. This will complete the proof since $\chi_{\rho}(H) < \chi_{\rho}(G)$ because $\rho$ is a multiplicative invariant. 

    \begin{claim}\label{claim about index}
        $k = k' = \frac{1-r-n}{1-m}$.
    \end{claim}
    \begin{proof}[Proof of Claim.] 
        We prove the claim for $k$; the proof for $k'$ is analogous. First, by Proposition~\ref{prop:prime_mult},
            \begin{eqnarray}
                \chi_{\rho}(G') &=& k\chi_{\rho}(G) \\
                 &=& k \left( 1-m + \sum_{i=1}^m \rho(G_i) \right).
            \end{eqnarray}
        Next, by the definition of $\chi_{\rho}$,
           \begin{eqnarray}
               \chi_{\rho}(G') &=& 1-r-n + \sum_{i=1}^n \rho(G_i').
           \end{eqnarray}
                           
        For $1 \leq i \leq m$, let $K_i = \{G_{i_j}'\}_{j=1}^{n_i}$ and $i_j \in \{1, \ldots, n\}$ be the collection of free factors of $G'$ so that $G_{i_j}'$ covers $G_i$. Then, for all $1 \leq i \leq m$, \[ \sum_{j=1}^{n_i} \rho(G_{i_j}') = k\rho(G_i) \]
        by Observation~\ref{obs:multpreimage}.
        Since the $K_i$ partition the set $\{G_i'\}_{i=1}^n$, Equation~(3) can be rewritten as 
        \begin{eqnarray}
            \chi_{\rho}(G') &=& 1-r-n + \left( \sum_{j=1}^{n_1} \rho(G_{i_j}') + \ldots + \sum_{j=1}^{n_m} \rho(G_{i_j}')  \right) \\
            &=& 1-r-n + k\sum_{i=1}^m \rho(G_i).
        \end{eqnarray}
        Combining Equation~(2) and Equation~(5) yields
        \begin{eqnarray}
          k \left( 1-m + \sum_{i=1}^m \rho(G_i) \right)  &=& 1-r-n + k\sum_{i=1}^m \rho(G_i).
        \end{eqnarray}
        Thus, $k = \frac{1-r-n}{1-m}$. 
    \end{proof}

    Therefore, $k=k'$, completing the proof of the theorem. 
\end{proof}

\begin{rem}
Note that in the proof of Claim~\ref{claim about index} we have not used that the multiplicative invariant $\rho$ is non-zero. In fact, the same proof works for the multiplicative invariant $\rho(G)=1/|G|$ (where $1/\infty =0$).
\end{rem}

\begin{rem} (Weakening the hypotheses.)
    If $G = G_1 * \ldots * G_m$ and $H = H_1 * \ldots * H_m$ are in the family $\cC'$ and $\chi_{\rho}(G)>\chi_{\rho}(H)$, then the proof of the theorem above shows that $G$ and $H$ are not abstractly commensurable. So, for example, one can weaken the hypothesis of the theorem to asking that $H_i \leq G_i$ is a {\it proper} finite-index subgroup for only one $i \in \{1, \ldots, n\}$. 
\end{rem}

    The general infinite-ended case now follows from Theorem~\ref{thm:mainpaper}. 

\begin{proof}[Proof of Corollary~\ref{main cor}]
    Suppose $G$ is a finitely presented infinite-ended group. Work of Stallings and Dunwoody~\cite{stallings,dunwoody1985accessibility} proves that $G$ is the fundamental group of a finite graph of groups with finite edge groups and vertex groups that have at most one end. If $G$ is not virtually free then there is at least one one-ended factor in this decomposition. Suppose the one-ended vertex groups in this decomposition are $G_1, \ldots, G_n$ for some $n\geq 1$ and belong to a family of groups $\cC$ that is closed under taking finite-index subgroups. Suppose $\cC$ has a multiplicative invariant that is nonzero for every element of $\cC$. Suppose further that each one-ended factor of $G$ contains a non-trivial finite-index subgroup. 
    The group $G$ is quasi-isometric to the free product $G_1 * G_1 * \ldots * G_n$ by work of Papasoglu--Whyte~\cite{papasogluwhyte}. Thus, $G$ is not quasi-isometrically rigid by Theorem~\ref{thm:mainpaper}.     
\end{proof}


\bibliographystyle{alpha}
\bibliography{mybib}

\end{document}